\documentclass[12pt]{article}
\usepackage{geometry}                
\usepackage{graphicx}
\usepackage{amssymb}
\usepackage{amsmath}
\usepackage{mathtools}
\usepackage{mathrsfs}
\usepackage{xcolor}
\usepackage{amsthm}
\usepackage{hyperref}
\newcommand*{\FF}{\mathbb{F}}
\newcommand*{\NN}{\mathbb{N}}
\newcommand*{\ZZ}{\mathbb{Z}}

\newcommand*{\RR}{\mathbb{R}}

\newcommand*{\calO}{\mathcal{O}}
\newcommand*{\calU}{\mathcal{U}}

\newcommand*{\tail}{\mathscr{T}}
\newcommand*{\scrS}{\mathscr{S}}

\newcommand*{\A}{\mathtt{A}}
\newcommand*{\ta}{\mathtt{a}}

\newcommand*{\tb}{\mathtt{b}}

\newcommand*{\1}{\mathbf{1}}

\newcommand*{\mb}[1]{\mathbf{#1}}
\newcommand*{\bOmega}{\mb{\Omega}}

\newcommand{\PA}{\textsc{pa}}

\newcommand*{\st}{\,:\,}

\newcommand*{\ball}[3][\relax]{\mathrm{B}^{#1}(#2, #3)}

\newtheorem{lemma}{Lemma}[section]
\newtheorem{cor}[lemma]{Corollary}
\newtheorem{prop}[lemma]{Proposition}
\newtheorem{theorem}[lemma]{Theorem}

\newtheorem{mainthm}{Theorem}

\theoremstyle{definition}

\DeclareMathOperator{\Prob}{Prob}

\DeclareMathOperator{\Unif}{Unif}

\DeclareMathOperator{\Hom}{Hom}

\DeclareMathOperator{\Sym}{Sym}

\DeclareMathOperator{\Lip}{Lip}

\DeclareMathOperator{\TV}{TV}

\DeclareMathOperator{\shent}{H}
\DeclareMathOperator{\fed}{a}

\DeclarePairedDelimiter{\abs}{\lvert}{\rvert}

\DeclarePairedDelimiter{\norm}{\|}{\|}
\newcommand{\nnorm}[1]{{\left\vert\kern-0.25ex\left\vert\kern-0.25ex\left\vert #1 
    \right\vert\kern-0.25ex\right\vert\kern-0.25ex\right\vert}}
\DeclarePairedDelimiterX{\inprod}[2]{\langle}{\rangle}{#1,\ #2}

\title{Free Energy, Gibbs Measures, and Glauber Dynamics for Nearest-neighbor Interactions on Trees}
\author{Christopher Shriver}

\begin{document}
\maketitle

\begin{abstract}
We extend results of R.~Holley beyond the integer lattice to a large class of groups which includes free groups. In particular we show that a shift-invariant measure is Gibbs if and only if it is Glauber invariant. Moreover, any shift-invariant measure converges weakly to the set of Gibbs measures when evolved under Glauber dynamics. These results are proven using a new notion of free energy density relative to a sofic approximation by homomorphisms. Any measure which minimizes free energy density is Gibbs.
\end{abstract}

\section{Introduction, Main Results}

A prototypical example of the type of system we study here is the Ising model, an old and well-studied model of magnetism. In this model we have a rectangular grid of particles, each of which can have `spin' either $+1$ or $-1$. Each particle interacts only with its nearest neighbors: neighboring particles with opposite spins increase the energy and neighboring particles with the same spin decrease the energy. The system prefers to be in a low-energy state. \\

The rectangular grid is natural for modeling an arrangement of particles in euclidean space. However, it is also natural to study systems with similar pairwise interactions between particles but other dependence structures. One important feature of the rectangular lattice $\ZZ^r$ that we would like to keep is its notion of `translation.' It is also important that each vertex has only finitely many neighbors. To preserve these features we generalize by replacing the rectangular lattice with the Cayley graph of a finitely-generated group. Some of our results will require the additional assumption that the group has ``property \PA'' (a definition is given in Section \ref{sec:freeenergy}).

In this setting we can model an infinite regular tree as the graph of the free group $\FF_r$ (which produces a $2r$-regular tree) or of the free product of cyclic groups $(\ZZ_2)^{*r}$ (which produces an $r$-regular tree). Note that these result in different notions of translation. \\

%
%
%
%
%

In the present paper we focus on extending results of Holley in \cite{holley1971}. He studied a natural notion of free energy density for systems with sites indexed by $\ZZ^r$, and used it to relate Gibbs measures and Glauber dynamics. His approach does not seem to work for nonamenable groups due to non-negligibility of the boundary of large finite subsystems.

In its place we use an ``extrinsic'' approach to free energy density which is inspired by recent work on the entropy theory of nonamenable group actions, initiated by Lewis Bowen \cite{bowen2010b} to solve similar problems which appear in that area. 

\subsection{Related work}

In one respect, Holley \cite{holley1971} worked in slightly more generality than we do here: he considered finite-range interactions, not just nearest-neighbor interactions. Higuchi and Shida \cite{higuchi1975} extended his results to spin systems on $\ZZ^r$ which may have infinite-range interactions, but the strength of the interactions is assumed to decay sufficiently quickly.

The method of the present paper may be compatible with such generalizations, but for the sake of simplicity we choose not to pursue them here.

More recently, Jahnel and K\"ulske \cite{jahnel2019} have extended the free energy density approach to non-reversible dynamics on integer-lattice systems. \\

Caputo and Martinelli \cite{caputo2006} have shown that if we evolve the product of plus-biased Bernoulli measures by Ising Glauber dynamics on an infinite tree, then it converges weakly to the ``plus boundary conditions'' Gibbs measure. \\

There has been some other work on notions of free energy density for Ising models on nonamenable groups, but these notions do not appear to have the properties we want for our present purposes. Dembo and Montanari \cite{dembo2010} consider, as we do below, a sequence of finite graphs that locally converge to an infinite tree. Their work differs from ours in that they study the limiting free energy density of the (unique) Gibbs measures on these finite graphs, while we study the free energy density of finitary measures which are locally consistent with a chosen infinitary measure (which is not necessarily Gibbs).

\subsection{Precise statements of basic definitions and main theorems}

Let $\Gamma$ be a countable group with $r$ generators $s_1, \ldots, s_r$ and arbitrary relations. Let $e \in \Gamma$ denote the identity. We will identify $\Gamma$ with its left Cayley graph, which has vertex set $\Gamma$ and an $i$-labeled directed edge $(\gamma, s_i \gamma)$ for every $i \in [r]$ and $\gamma \in \Gamma$.

For some finite alphabet $\A$, we define the action of $\Gamma$ on $\A^\Gamma$ by
	\[ (\beta \mb{y}) (\gamma) = \mb{y}(\gamma \beta) \]
for $\beta, \gamma \in \Gamma$. We can think of this as moving the center of the labeling to $\beta^{-1}$. We say that a measure $\mu \in \Prob(\A^\Gamma)$ is shift-invariant if $\beta_*\mu = \mu$ for any $\beta \in \Gamma$, where $\beta_*$ denotes the pushforward. We denote the set of shift-invariant probability measures by $\Prob^\Gamma (\A^\Gamma)$.

If $V$ is a finite set, we can consider the set $\Hom(\Gamma, \Sym(V))$ of homomorphisms from $\Gamma$ to the group of permutations of $V$. It is possible for this set to be empty. Given $\sigma \in \Hom(\Gamma, \Sym(V))$, we write the permutation which is the image of $\gamma \in \Gamma$ by $\sigma^\gamma$. We can associate to $\sigma$ a directed graph with vertex set $V$ and an $i$-labeled edge $(v, \sigma^{s_i}(v))$ for each $i \in [r]$ and $v \in V$.

The graph of any $\sigma$ can be thought of as a finite system which locally looks like $\Gamma$, just as a large rectangular grid locally looks like the integer lattice $\ZZ^r$.

Either $\Gamma$ or the graph of some $\sigma$ can be endowed with a natural graph distance: the distance between a pair of vertices is defined to be the minimal number of edges in a path between them, ignoring edge directions. Let $\ball[\sigma]{v}{R}$ denote the closed radius-$R$ ball centered at $v \in V$, and similarly define $\ball[\Gamma]{\gamma}{R}$ for $\gamma \in \Gamma$.

The correspondence between finite and infinite systems is established using \emph{empirical distributions}, which we now define. Let $\sigma \in \Hom(\Gamma, \Sym(V))$ and $\mb{x} \in \A^V$. For any $v \in V$ there is a natural way to lift $\mb{x}$ to a labeling $\Pi_v^\sigma \mb{x} \in \A^\Gamma$, starting by lifting $\mb{x}_v$ to the root $e$. More precisely,
	\[ \big( \Pi_v^\sigma \mb{x} \big) (\gamma) = \mb{x}\big( \sigma^\gamma (v) \big). \]
The empirical distribution of $\mb{x}$ is defined by
	\[ P_{\mb{x}}^\sigma = (v \mapsto \Pi_v^\sigma \mb{x})_* \Unif(V) = \frac{1}{\abs{V}} \sum_{v \in V} \delta_{\Pi_v^\sigma \mb{x}} \in \Prob\!\big( \A^\Gamma \big) . \]
This captures the `local statistics' of $\mb{x}$. This notation was used in the approach to sofic entropy in \cite{austin2016}.
	
To state our results we use the following way of measuring local similarity of a finite graph $\sigma$ to $\Gamma$: for $\sigma \in \Hom(\Gamma, \Sym(V))$, we say $\ball[\sigma]{v}{R} \cong \ball[\Gamma]{e}{R}$ if there is an isomorphism of the induced subgraphs $\ball[\sigma]{v}{R}, \ball[\Gamma]{e}{R}$ which respects both edge labels and directions. Define
\begin{align*}
	\delta_R^\sigma &= \frac{1}{\abs{V}} \abs{\{v \in V \st \ball[\sigma]{v}{R} \not\cong \ball[\Gamma]{e}{R}\}} \\
	\Delta^\sigma &= \inf_R \big(9 \cdot (2/3)^{R} + 6\delta_R^\sigma \big).
\end{align*}
The constants which appear here are connected to the choice of metric $d$ in Section~\ref{sec:definitions} below. If $\Delta^\sigma$ is small, then $\sigma$ looks like $\Gamma$ to a large radius near most vertices. This is a way of saying that the action of $\sigma$ is approximately free. Note also that a sequence $\sigma_n$ Benjamini-Schramm converges to the infinite tree $\Gamma$ (or equivalently is a sofic approximation to the group $\Gamma$) if and only if $\Delta^{\sigma_n} \to 0$.

As mentioned above, our central tool is a notion of ``free energy density'' of a measure $\mu \in \Prob(\A^\Gamma)$ with respect to a nearest-neighbor potential. This free energy density is defined relative to a choice of $\Sigma = (\sigma_n \in \Hom(\Gamma, \Sym(V_n))_{n \in \NN}$ with $\Delta^{\sigma_n} \to 0$. It may be $+\infty$, but if it is finite then it is nonincreasing as $\mu$ evolves according to Glauber dynamics (Proposition~\ref{prop:nonincreasing}). Moreover if $\mu$ is not Gibbs then it is strictly decreasing; Proposition~\ref{prop:strictdecrease} gives a stronger version of this claim. For every choice of $\Sigma$ there exist measures with finite free energy density, so this implies the following:

\begin{mainthm}
	For any choice of $\Sigma$, every $\mu \in \Prob(\A^\Gamma)$ minimizing the free energy density relative to $\Sigma$ is Gibbs.
\end{mainthm}

The converse is false, since a Gibbs measure may have free energy density $+\infty$ with respect to some $\Sigma$. It is unclear whether a Gibbs measure may have finite but non-minimal free energy density.

\begin{mainthm}
\label{thm:main}
	Suppose $\mu \in \Prob^\Gamma(\A^\Gamma)$, and let $\mu_t$ denote its evolution under Glauber dynamics. If there exist $s \geq 0$ and $\Sigma$ such that $\fed_\Sigma(\mu_s) < +\infty$, then $\mu_t$ converges weakly to the set of Gibbs measures as $t \to \infty$.
\end{mainthm}
It is possible to avoid the degenerate case of infinite free energy density by an appropriate choice of $\Sigma$  when $\Gamma$ has a property called ``property \PA''; see Section \ref{sec:freeenergy} below for a definition and the relevant result (Proposition~\ref{prop:PAequiv}). Hence we have the following:
\begin{cor}
	If $\Gamma$ has property \PA, then a shift-invariant measure is Gibbs if and only if it is Glauber-invariant.
\end{cor}


\subsection{Acknowledgements}

Thanks to Tim Austin for many helpful conversations and for providing feedback on a preprint, and to Lewis Bowen for telling me about property \PA\ and its relation to sofic approximations. Thanks also to Raimundo Brice\~no for helpful correspondence regarding the Gibbness of free energy minimizers.

This material is based upon work supported by the National Science Foundation under Grant No.~DMS-1855694.

\section{Definitions}
\label{sec:definitions}

For $\gamma \in \Gamma$, let $\abs{\gamma}$ denote the graph distance between $\gamma$ and the identity $e \in \Gamma$.

Give $\A^\Gamma$ the metric
	\[ d(\mathbf{x},\mathbf{y}) = \sum_{\gamma \in \Gamma} (3r)^{-\abs{\gamma}} \1_{\mathbf{x}(\gamma) \ne \mathbf{y}(\gamma)} ; \]
the factor 3 is chosen to ensure convergence.
Let $\bar{d}$ denote the corresponding transportation metric on $\Prob(\A^\Gamma)$; specifically, with $\Lip_1(\A^\Gamma)$ denoting the set of 1-Lipschitz real-valued functions, we define
	\[ \bar{d}(\mu, \nu) = \sup\left\{ \abs*{ \mu f - \nu f} \st f \in \Lip_1(\A^\Gamma) \right\} . \]
Here $\mu f$ denotes the integral of $f$ with respect to $\mu$.
The sum defining $d$ always converges because every vertex of $\Gamma$ has at most $2r$ neighbors. Note that $d$ generates the product topology (which is compact), and $\bar{d}$ generates the weak topology induced by the pairing with continuous functions (which is also compact). \\

For any set $V$ and any $\mb{x} \in \A^V$, $v \in V$, $\ta \in \A$ we let $\mb{x}^{v \to \ta} \in \A^V$ be given by
	\[ \mb{x}^{v \to \ta}(w) = \left\{
						\begin{array}{ll}
							\mb{x}(w), 	& w \ne v \\
							\ta,			& w = v .
						\end{array} \right. \]
Below, an element of $\A^V$ will be referred to as a \emph{microstate} and an element of $\Prob(\A^V)$ as a \emph{state}.

\subsection{Glauber dynamics}

Let $V$ be an at most countable set and fix $\sigma \in \Hom(\Gamma, \Sym(V))$. We will apply this in two cases: when $V$ is finite, and when $V = \Gamma$ and $\sigma$ is the action of $\Gamma$ on itself by right multiplication. We will distinguish between these cases by giving notation superscripts of $\sigma$ or $\Gamma$ respectively (e.g. $\Omega^\sigma$ versus $\Omega^\Gamma$).

For simplicity we consider only nearest-neighbor interactions. Fix a symmetric function $J \colon \A^2 \to \RR$ and a function $h \colon \A \to \RR$, and let $S = \{ s_1, \ldots, s_r, s_1^{-1}, \ldots, s_r^{-1}\}$.
For $v \in V$, $\mb{x} \in \A^V$, let
	\[ \Phi_v(\mb{x}) = h(\mb{x}(v)) + \sum_{s \in S} J(\mb{x}(v), \, \mb{x}(\sigma^s v)) , \]
and for $\ta \in \A$ let
	\[ c_v(\mb{x}, \ta) = Z_v(\mb{x})^{-1} \exp \!\left\{ - \Phi_v(\mb{x}^{v \to \ta}) \right\} , \]
where $Z_v(\mb{x})$ is the normalizing factor which makes $c_v(\mb{x}, \cdot)$ a probability measure on $\A$. Note that $c_v(\mb{x}, \ta)$ only depends on $\mb{x}$ through its values at vertices adjacent to $v$.
A standard stochastic Ising model has $\A = \{-1, 1\}$ and $J(\ta, \tb) = \beta \ta \tb$ for some $\beta \geq 0$ (the inverse temperature), and $h$ represents an ambient electric field.

We study here the continuous-time Markov process with state space $\A^V$ and generator given by
	\[ \Omega f (\mb{x}) = \sum_{v \in V} \sum_{\ta \in \A} c_v(\mb{x}, \ta) [f(\mb{x}^{v \to \ta}) - f(\mb{x})] . \]
If $V$ is finite then this gives a well-defined linear operator on $C(\A^V)$. Otherwise we need to first define $\Omega$ using the above formula on a `core' of `smooth' functions for which the sum converges, then take the closure of $\Omega$; see Liggett's book \cite{liggett2005} for details. We denote the induced Markov semigroup by $\{S(t) \st t \geq 0\}$.

For any continuous function $f \colon \A^V \to \RR$ and $\mb{x} \in \A^V$, we interpret $[S(t) f](\mb{x})$ as expected value of $f(\mb{x}_t)$, where the random variable $\mb{x}_t$ is the evolution of $\mb{x}$ to time $t$. The action of the semigroup on a probability measure is denoted $\mu S(t)$; this is interpreted the evolution of a probability measure $\mu \in \Prob(\A^V)$ to time $t$. The evolved measure is defined via the formula $[\mu S(t)] f \coloneqq \mu [S(t) f]$.\\

Further details of the construction of the dynamics will only be needed for proofs of the following two results. The relevant details are contained in Section \ref{sec:infdynamics}.

The first result may be thought of as an approximate equivariance between the Glauber semigroups and the empirical distribution:

\begin{theorem}
\label{thm:Sequivariance}
There is a constant $M > 0$ such that for any $\mb{x} \in \A^V$, $\sigma \in \Hom(\Gamma, \Sym(V))$, and $t\geq 0$
	\[ \bar{d}\left(S^\sigma(t) P_{\mb{x}}^\sigma,\ P_{\mb{x}}^\sigma S^\Gamma(t) \right) \leq \Delta^\sigma \cdot t e^{Mt} . \]
\end{theorem}
It may be helpful to clarify that the first term on the left, $S^\sigma(t) P_{\mb{x}}^\sigma$, is the evolution to time $t$ of the function $P_{\bullet}^\sigma \colon \A^V \to \Prob(\A^\Gamma)$ evaluated at $\mb{x} \in \A^V$. The second term is the evolution of the empirical distribution $P_{\mb{x}}^\sigma$. So this theorem says that the expected empirical distribution after running the finitary dynamics for time $t$ is close to the result of evolving the original empirical distribution for time $t$, as long as $\sigma$ locally looks like $\Gamma$.

We also use the following Lipschitz bound on the Markov semigroup:

\begin{lemma}
\label{lem:misc}
	If $\mu, \nu \in \Prob(\A^\Gamma)$ then
		\[ \bar{d}\big( \mu S^\Gamma(t),\, \nu S^\Gamma(t)\big) \leq \exp(Mt) \, \bar{d}(\mu, \nu) . \]
\end{lemma}

\subsection{Gibbs measures}

For a finite graph with vertex set $V$ induced by $\sigma \in \Hom(\Gamma, \Sym(V))$, we define the total energy
	\[ U(\mb{x}) = \sum_{v \in V} h(\mb{x}(v)) + \sum_{v \in V} \sum_{i \in [r]} J(\mb{x}(v), \mb{x}(\sigma^{s_i} v)) . \]
Note that each directed edge appears in the double sum exactly once; in particular, this is not the same as $\sum_v \Phi_v (\mb{x})$, in which directed edges will be counted twice. Note that if we define
	\[ u^{\max} = \max_{\ta \in \A} \left( h(\ta) + r \max_{\tb \in \A} J(\ta, \tb) \right) \]
and
	\[ u^{\min} = \min_{\ta \in \A} \left( h(\ta) + r \min_{\tb \in \A} J(\ta, \tb) \right) \]
then for any $V, \sigma$ and any $\mb{x} \in \A^V$ we have
	\[ u^{\min} \leq \frac{1}{\abs{V}} U(\mb{x}) \leq u^{\max} . \]
	
The finitary Gibbs measure $\xi_V \in \Prob(\A^V)$ is defined by
	\[ \xi_V \{\mb{x}\} = Z_V^{-1} \exp \{ - U(\mb{x}) \}  \]
where $Z_V$ is the normalizing constant. Note that
	\[ \xi_v(\mb{y}(v) = \ta \mid \mb{y}(w) =\mb{x}(w)\ \forall w \ne v) = \frac{\exp\{-U(\mb{x}^{v \to \ta})\}}{\sum_{\tb \in \A} \exp\{-U(\mb{x}^{v \to \tb})\}} = c_v(\mb{x}, \ta) , \]
since
	\[ \frac{\exp\{-U(\mb{x}^{v \to \ta})\}}{\exp\{-U(\mb{x}^{v \to \tb})\}} = \frac{\exp\{-\Phi_v(\mb{x}^{v \to \ta})\}}{\exp\{-\Phi_v(\mb{x}^{v \to \tb})\}} . \]

On the infinite graph $\Gamma$ we must use a different approach, since the sum defining the total energy will not converge. We use a natural generalization of \cite[Definition IV.1.5]{liggett2005}; see also \cite{georgii2011} for a much more general treatment of infinite-volume Gibbs measures.

Let $\tail_\gamma$ denote the $\sigma$-algebra generated by all vertices except for $\gamma$. We call $\mu \in \Prob(\A^\Gamma)$ a Gibbs measure if for each $\gamma \in \Gamma$ and $\ta \in \A$, the function $\mb{y} \mapsto c_\gamma(\mb{y}, \ta)$ is a version of the conditional expectation $\mu(\{\mb{x} \st \mb{x}(\gamma) = \ta\} \mid \tail_v)(\mb{y})$. This means that for every integrable $f \colon \A^\Gamma \to \RR$ and $\gamma \in \Gamma$ we have
	\[ \int \sum_{\ta \in \A} c_\gamma( \mb{x}, \ta) f(\mb{x}^{\gamma \to \ta}) \mu(d\mb{x}) = \int f(\mb{x})\, \mu(d\mb{x}) . \]
We may also describe this relation by saying that $\mu$ is invariant under re-randomizing the spin at $\gamma$ using the kernel $c_\gamma$.

In particular, if $\mu$ is Gibbs then for any `smooth' $f$
	\[ \mu \Omega^\Gamma f = \int \left( \sum_{\gamma,\ta} c_\gamma( \mb{x}, \ta) [ f(\mb{x}^{\gamma \to \ta}) - f(\mb{x})]\right) \mu(d\mb{x}) = 0 . \]
It follows that $\mu \Omega^\Gamma = 0$, which means Gibbs measures are Glauber-invariant. \\


\subsection{Good models for measures on $\A^\Gamma$}

Let $V$ be a finite set and let $\sigma \in \Hom(\Gamma, \Sym(V))$. A labeling $\mb{x} \in \A^V$ is said to be a \emph{good model} for $\mu \in \Prob(\A^\Gamma)$ over $\sigma$ if the empirical distribution $P_{\mb{x}}^\sigma$ is close to $\mu$ in the weak topology. More precisely, we can say $\mb{x}$ is $\calO$-good if $P_{\mb{x}}^\sigma \in \calO$ for some weak-open neighborhood $\calO \ni \mu$. The set of such $\mb{x}$ is denoted $\Omega(\sigma, \calO)$. An interpretation of this relationship is that average local quantities of the finite system are consistent with $\mu$.


We define the empirical distribution of a state $\zeta \in \Prob(\A^V)$ by
	\[ P_\zeta^\sigma \coloneqq \zeta P_{\mb{x}}^\sigma \left( = \int P_{\mb{x}}^\sigma\, \zeta(d\mb{x}) \right) \]
and say that $\zeta$ is $\calO$-consistent with $\mu$ (for some neighborhood $\calO\ni \mu$) over $\sigma$ if $P_\zeta^\sigma \in \calO$. We can still interpret this in terms of averages of local quantities: now the average also involves a random microstate $\mb{x}$ with law $\zeta$. We denote the set of such states by $\bOmega(\sigma, \calO)$. This measure essentially appears in the notion of ``local convergence on average'' introduced in \cite[Definition 2.3]{montanari2012}.



This consistency is stable under Glauber dynamics in the following sense:

\begin{prop}
\label{prop:stability}
	Suppose $\sigma \in \Hom(\Gamma, \Sym(V))$, $\zeta \in \Prob(\A^V)$, and $\mu \in \Prob(\A^\Gamma)$. Let $\zeta_t, \mu_t$ denote their evolutions under Glauber dynamics on $\sigma, \Gamma$ respectively. Then for any $t \geq 0$
		\[ \bar{d} \big( P_{\zeta_t}^\sigma, \mu_t \big) \leq \big[ \Delta^\sigma t + \bar{d} \big( P_\zeta^\sigma,\, \mu \big) \big]\exp(Mt) .\]
\end{prop}
We give the proof here, since it uses only results already stated:
\begin{proof}
	For any $f \in \Lip_1(\A^\Gamma)$, the triangle inequality gives
	\begin{align*}
		\abs*{ P_{\zeta_t}^\sigma f - \mu_t f}
			&\leq \abs*{P_{\zeta_t}^\sigma f - P_\zeta^\sigma S^\Gamma(t) f} + \abs*{P_\zeta^\sigma S^\Gamma(t) f - \mu_t f} \\
			&= \abs*{\zeta \big[S^\sigma (t) P_{\mb{x}}^\sigma f - P_{\mb{x}}^\sigma S^\Gamma(t) f\big]} + \abs*{P_\zeta^\sigma S^\Gamma(t) f - \mu S^{\Gamma}(t) f}.
	\intertext{Using that $\zeta$ is a probability measure and the definition of $\bar{d}$, this implies the bound}
		\abs*{ P_{\zeta_t}^\sigma f - \mu_t f}
			&\leq \max_{\mb{x} \in \A^V} \bar{d} \big( S^\sigma(t) P_{\mb{x}}^\sigma,\, P_{\mb{x}}^\sigma S^\Gamma(t) \big) + \bar{d} \big( P_\zeta^\sigma S^\Gamma(t),\, \mu S^\Gamma(t) \big).
	\intertext{The first term may be controlled with Theorem~\ref{thm:Sequivariance} and the second with Lemma~\ref{lem:misc} to get}
		\abs*{ P_{\zeta_t}^\sigma f - \mu_t f}
			&\leq \big[ \Delta^\sigma t + \bar{d} \big( P_\zeta^\sigma,\, \mu \big) \big]\exp(Mt) .
	\end{align*}
	The result then follows by taking the supremum over $f \in \Lip_1(\A^\Gamma)$.
\end{proof}

\subsection{Free energy density}
\label{sec:freeenergy}
Given a sequence $\Sigma = (\sigma_n)_{n \in \NN}$ with $\sigma_n \in \Hom(\Gamma, \Sym(V_n))$ such that $\abs{V_n} \to \infty$ and $\Delta^{\sigma_n} \to 0$, we define the \emph{free energy density} of $\mu \in \Prob^\Gamma(\A^\Gamma)$ relative to $\Sigma$ by
	\[ \fed_\Sigma(\mu) = \lim_{\calO \downarrow \mu} \limsup_{n \to \infty} \inf_{\zeta \in \bOmega(\sigma_n, \calO)} \frac{1}{\abs{V_n}} \big[ \zeta (U) - \shent(\zeta) \big] \]
where $\zeta(U)$ is the average energy and $\shent(\zeta)$ is the Shannon entropy, and we follow the convention that the infimum of the empty set is $+\infty$.

The limit is over the net of weak-open neighborhoods of $\mu$, partially ordered by inclusion. Note that for each $n$ the expression $\inf_{\zeta \in \bOmega(\sigma_n, \calO)} \frac{1}{\abs{V_n}} \big[ \zeta (U) - \shent(\zeta) \big]$ is nondecreasing as $\calO \downarrow \mu$, so the limit exists and is equal to the supremum over $\calO \ni \mu$.

Since the map $\mu \mapsto \fed_\Sigma(\mu)$ is defined in terms of a supremum over neighborhoods of $\mu$, it is lower semi-continuous. Consequently, it attains its minimum on $\Prob^\Gamma(\A^\Gamma)$.

Note also that as long as $\bOmega(\sigma_n, \calO)$ is nonempty, for every $\zeta$ in this set we have
	\[ u^{\min} - \log \abs{\A} \leq \frac{1}{\abs{V_n}} \big[ \zeta (U) - \shent(\zeta) \big] \leq u^{\max} . \]
In particular,
	\[ \fed_\Sigma(\mu) \in \big[ u^{\min} - \log \abs{\A}, \ u^{\max} \big] \cup \{+\infty\} . \]

The case $\fed_\Sigma(\mu) = +\infty$ can actually occur, for example if $\mu$ is ergodic and the sofic entropy relative to $\Sigma$ is $-\infty$. This is because, in the ergodic case, $\zeta$ being consistent with $\mu$ is the same as being mostly supported on labelings which are good models for $\mu$, but if the sofic entropy is $-\infty$ then there are no good models.

Note, however, that the function $\mu \mapsto \fed_\Sigma(\mu)$ is not identically $+\infty$ for any choice of $\Sigma$, since the point mass at a constant labeling in $\A^\Gamma$ always has good models.

If $\mu$ were not shift-invariant then the expression defining $\fed_\Sigma(\mu)$ would still make sense, but would take the value $+\infty$ for any $\Sigma$. This is because empirical distributions are always shift-invariant and $\Prob^\Gamma(\A^\Gamma)$ is closed so, no matter what $\Sigma$ we choose, any small enough neighborhood of $\mu$ contains no empirical distributions. In fact, we will see below that for some $\Gamma$ (for example $\Gamma = \FF_2 \times \FF_2$) there exist shift-invariant measures which cannot be approximated by empirical distributions over any $\Sigma$. In these cases the obstruction is that empirical distributions always have finite support.

Following \cite{bowen2003}, a shift-invariant measure in $\Prob^\Gamma(\A^\Gamma)$ is called \emph{periodic} if it has finite support, and a group $\Gamma$ is said to have \emph{property \PA} if the set of periodic measures is dense in $\Prob^\Gamma(\A^\Gamma)$ for every finite alphabet $\A$; in other words, if every shift-invariant probability measure on $\A^\Gamma$ has \textsc{p}eriodic \textsc{a}pproximations.

In Section \ref{sec:PAequiv} we prove the following:
\begin{prop}
\label{prop:PAequiv}
	A group $\Gamma$ has property \PA\ if and only if for every finite alphabet $\A$ and every $\mu \in \Prob^\Gamma(\A^\Gamma)$ there exists a sequence $\Sigma = \big( \sigma_n \in \Hom(\Gamma, \Sym(V_n)) \big)_{n \in \NN}$ such that $\Delta^{\sigma_n} \to 0$ and $\fed_\Sigma(\mu) < +\infty$.
\end{prop}

Property \PA\ was proved to hold for free groups by Bowen in \cite[Theorem 3.4]{bowen2003}. Kechris later studied another equivalent property in \cite{kechris2012} which he called ``\textsc{md}''; see also the survey \cite{burton2020} for more recent information on which groups are known to have this property. In particular:
\begin{itemize}
	\item Amenable groups have property \PA.
	\item If $\Gamma_1, \Gamma_2$ are both nontrivial groups which are either finite or property-\PA, then the free product $\Gamma_1 \ast \Gamma_2$ has property \PA. 
	\item The recent negative solution to Connes' embedding conjecture \cite{ji2020} implies that the direct product $\FF_2 \times \FF_2$ does \emph{not} have property \PA.
\end{itemize}

\subsubsection{Calculation}
In some cases it can be possible to calculate the free energy density. Fix a sequence $\Sigma$ with $\Delta^{\sigma_n} \to 0$, and for each $n$ let $\xi_n \in \Prob(\A^{V_n})$ denote the unique finitary Gibbs measure. Suppose that $P_{\xi_n}^{\sigma_n} \xrightarrow{wk} \mu$. Then for any $\calO \ni \mu$ we have $\xi_n \in \mb\Omega(\sigma_n, \calO)$ for all large enough $n$. But $\xi_n$, by virtue of being the Gibbs measure, has minimal free energy among all probability measures on $\A^{V_n}$, so 
	\[ \inf_{\zeta \in \mb\Omega(\sigma_n, \calO)} \big[ \zeta(U) - \shent(\zeta) \big] = \xi_n(U) - \shent(\xi_n) = -\log Z_n \]
where $Z_n$ is the normalizing constant appearing in the definition of $\xi_n$. It follows that
	\[ \fed_\Sigma(\mu) = - \liminf_{n \to \infty} \frac{1}{\abs{V_n}} \log Z_n . \]
	
The main theorem of \cite{montanari2012} implies that, for the Ising model with no external field on a regular tree, the weak limit of $P_{\xi_n}^{\sigma_n}$ is $\frac{1}{2}(\mu_+ + \mu_-)$ for any $\Sigma$, where $\mu_+,\mu_-$ are the Gibbs measures with $+,-$ boundary conditions respectively. In this case $\frac{1}{\abs{V_n}} \log Z_n$ actually converges, and the limit can be written down explicitly; see \cite{dembo2010}.

%

\subsection{Measuring non-Gibbs-ness}
As in \cite{holley1971}, we make use of the function 
	\[ F_0(s) = \left\{ \begin{array}{ll}
					s - s\log s - 1, & s > 0 \\
					-1, & s = 0
					\end{array} \right. \]
which appears in an expression for the time derivative of free energy [Proposition~\ref{prop:derivative}].
This function is concave, nonpositive, and equal to 0 if and only if $s=1$. A graph is included in Figure~\ref{fig:F0}.
\begin{figure}
\begin{center}
	\includegraphics[width=0.75\textwidth]{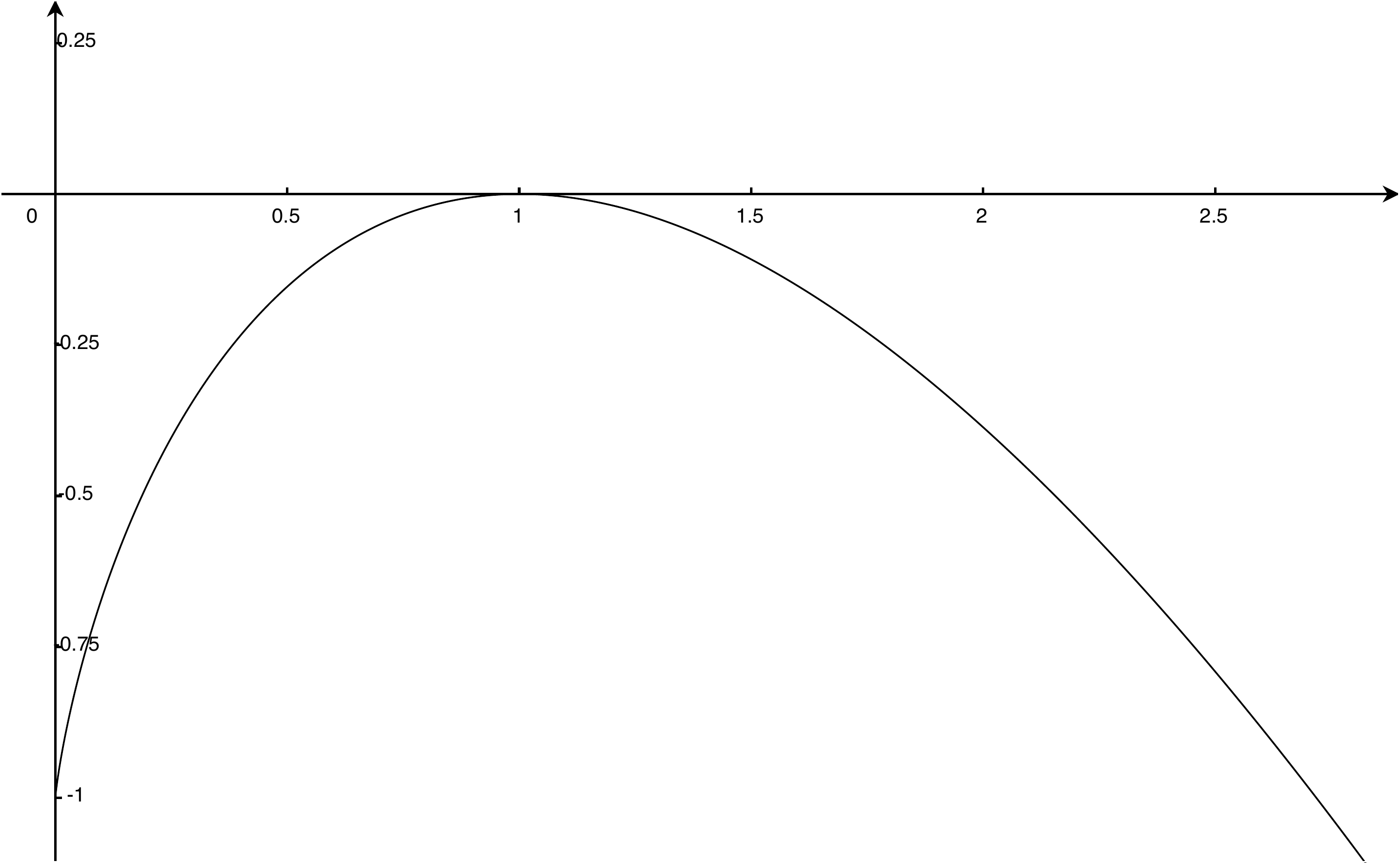}
	\caption{Graph of $F_0$}
	\label{fig:F0}
\end{center}
\end{figure}
If $\mu_R \in \Prob(\A^{\ball{e}{R}})$ has full support, we define
	\[ \Delta_\ta^R(\mu_R) = \sum_{\mb{y} \in \A^{\ball{e}{R}}} \mu_{R}\{\mb{y}\} \cdot F_0 \left( \frac{\exp\{ - \Phi_e(\mb{y})\}}{\exp\{ - \Phi_e(\mb{y}^{e \to \ta})\}} \frac{\mu_R \{\mb{y}^{e \to \ta}\}}{\mu_R \{\mb{y}\}}\right) . \]
This measures the average failure of $\mu_R$ to be consistent with the Gibbs specification.

\begin{lemma}
	Suppose $\mu \in \Prob(\A^\Gamma)$ is a translation-invariant measure such that for every $R \geq 1$ the marginal $\mu_R \in \Prob(\A^{\ball{e}{R}})$ has full support and $\Delta_\ta^R(\mu_R) = 0$ for every $\ta \in \A$. Then $\mu$ is Gibbs.
\end{lemma}
\begin{proof}
	Fix $R \geq 1$, and let $\scrS_{\ball{e}{R} \setminus \{e\}}$ denote the $\sigma$-algebra generated by sites in $\ball{e}{R} \setminus \{e\}$. Then by definition of conditional expectation, for any $\ta \in \A$
		\[ \mu\big(\{ \mb{x} \in \A^\Gamma \st \mb{x}(e) = \ta\} \mid \scrS_{\ball{e}{R} \setminus \{e\}} \big)(\mb{y}) = \frac{\mu_R\{\mb{y}^{e \to \ta}\}}{\sum_{\tb \in \A} \mu_R\{\mb{y}^{e \to \tb}\} } , \]
	where on the right-hand side we use the shorthand $\mu_R\{\mb{y}\} = \mu_R\{\mb{y} \!\restriction_{\ball{e}{R}}\}$ for $\mb{y} \in \A^\Gamma$.
	Our assumption that $\mu_R$ has full support and $\Delta_\ta^R(\mu_R) = 0$ implies that 
		\[ \frac{\mu_R\{\mb{y}^{e \to \ta}\}}{\sum_{\tb \in \A} \mu_R\{\mb{y}^{e \to \tb}\} }
			= \left( \sum_{\tb \in \A} \frac{\mu_R\{\mb{y}^{e \to \tb}\}}{\mu_R\{\mb{y}^{e \to \ta}\}} \right)^{-1}
			= \left( \sum_{\tb \in \A} \frac{\exp\{-\Phi_e(\mb{y}^{e \to \tb})\}}{\exp\{-\Phi_e(\mb{y}^{e \to \ta})\}} \right)^{-1}
			= c_e (\mb{y}, \ta) \]
	Taking $R$ to infinity, by martingale convergence we get
		\[ \mu(\cdot \mid \tail_e)(\mb{y}) = c_e(\mb{y}, \cdot ). \]
	By translation invariance, it follows that $\mu$ is Gibbs.
\end{proof}


\section{Proof of Theorem \ref{thm:main}}

\begin{prop}
\label{prop:derivative}
	Let $\zeta_0 \in \Prob(\A^V)$, and let $\zeta_t$ denote its evolution under Glauber dynamics. Then for all $t > 0$, $\zeta_t$ has full support and 
		\[ \frac{d}{dt} \big[ \zeta_t(U) - \shent(\zeta_t) \big]
			= \sum_{\mb{x}, v, \ta} F_0 \left( \frac{\exp\{ - \Phi_v(\mb{x})\}}{\exp\{ - \Phi_v(\mb{x}^{v \to \ta})\}} \frac{\zeta_t \{\mb{x}^{v \to \ta}\}}{\zeta_t \{\mb{x}\}}\right) \zeta_t \{\mb{x}\} c_v(\mb{x}, \ta) . \]
\end{prop}

Our proof of this proposition is based on the proof of the analogous result in Holley's paper \cite{holley1971}, with some minor changes.

For $\mb{x} \in \A^V$, write
	\[ P(\mb{x}) = \exp\{ -U(\mb{x})\} . \]
This is just the Gibbs measure on $V$, except without the normalizing factor. It is easy to see that
	\[ \zeta(U) - \shent(\zeta) = \sum_{\mb{x}} \zeta\{\mb{x}\} \log \frac{\zeta\{\mb{x}\}}{P(\mb{x})} . \]

\begin{proof}[Proof of proposition]
	A calculation using the Markov generator shows that
		\[ \frac{d}{dt} \zeta_t \{\mb{x}\} = \sum_{v,\ta} \big[ \zeta_t\{ \mb{x}^{v \to \ta}\} c_v(\mb{x}^{v \to \ta}, \mb{x}(v)) - \zeta_t \{\mb{x}\} c_v (\mb{x}, \ta) \big] . \]
	In particular, if $\mb{x}$ is such that $\zeta_t\{\mb{x}\} = 0$ but there exist $v,\ta$ such that $\zeta_t\{ \mb{x}^{v \to \ta}\} > 0$, then $\frac{d}{dt} \zeta_t \{\mb{x}\} > 0$. Unless $t=0$ this would imply the existence of times where $\zeta_s$ gives negative mass to $\{x\}$. Therefore $\zeta_t$ has full support for all $t > 0$. 
	
	Therefore
	\begin{align*}
		\frac{d}{dt} \big[ \zeta_t(U) - \shent(\zeta_t) \big]
			&= \sum_{\mb{x} \in \A^V} \frac{d}{dt} \left[ \zeta_t \{\mb{x}\} \log \frac{\zeta_t \{\mb{x}\}}{P(\mb{x})} \right] \\
			&= \sum_{\mb{x} \in \A^V} \frac{d}{dt} \left[ \zeta_t \{\mb{x}\} \right]  \log \frac{\zeta_t \{\mb{x}\}}{P(\mb{x})} \\
			&= \sum_{\mb{x} \in \A^V} \sum_{v,\ta} \big[ \zeta_t\{ \mb{x}^{v \to \ta}\} c_v(\mb{x}^{v \to \ta}, \mb{x}(v)) - \zeta_t \{\mb{x}\} c_v (\mb{x}, \ta) \big]  \log \frac{\zeta_t \{\mb{x}\}}{P(\mb{x})}.
	\end{align*}
	
	For $\mb{x}, \mb{y} \in \A^V$, define
		\[ \mathfrak{A}(\mb{x}, \mb{y})
			= \left\{ \begin{array}{cl}
				-\displaystyle\sum_{\substack{\ta, v \\ \ta \ne \mb{x}(v)}} c_v(\mb{y}, \ta) , & \mb{x} = \mb{y} \\
				c_v(\mb{y}, \ta) , & \mb{x} \ne \mb{y},\ \mb{x} = \mb{y}^{v \to \ta} \text{ for some } v \in V, \ta \in \A \\
				0 , & \text{else}.
			\end{array} \right. \]
	This has the following useful properties:
	\begin{equation}
		\sum_{\mb{x}} \mathfrak{A}(\mb{x}, \mb{y}) = 0 \quad \forall \mb{y}.
	\end{equation}
	\begin{proof}
	\renewcommand{\qedsymbol}{$\triangleleft$}
		For any $\mb{y}$,
			\[ \sum_{\mb{x}} \mathfrak{A}(\mb{x}, \mb{y})  = \mathfrak{A}(\mb{y}, \mb{y}) + \sum_{v \in V} \sum_{\ta \ne \mb{y}(v)} \mathfrak{A}(\mb{y}^{v \to \ta}, \mb{y}) + 0 = 0 . \qedhere \]
	\end{proof}
	\begin{equation}
		\sum_{\mb{y}} \mathfrak{A}(\mb{x}, \mb{y}) P(\mb{y}) = 0 \quad \forall \mb{x}.
	\end{equation}
	\begin{proof}
	\renewcommand{\qedsymbol}{$\triangleleft$}
		For any $\mb{x}$,
		\begin{align*}
			\sum_{\mb{y}} \mathfrak{A}(\mb{x}, \mb{y}) P(\mb{y})
				&= \mathfrak{A}(\mb{x}, \mb{x}) P(\mb{x}) + \sum_{v \in V} \sum_{\ta \ne \mb{x}(v)} \mathfrak{A}(\mb{x}, \mb{x}^{v \to \ta}) P(\mb{x}^{v \to \ta}) + 0 = 0 \\
				&= \sum_{\substack{\ta, v \\ \ta \ne \mb{x}(v)}}  \big[ -c_v(\mb{x}, \ta) \exp\{ - U(\mb{x})\} + c_v(\mb{x}^{v \to \ta}, \mb{x}(v)) \exp\{-U(\mb{x}^{v \to \ta})\} \big] \\
				&= 0 .
		\end{align*}
		In fact every term of this last sum is zero because
			\[ \frac{c_v(\mb{x}, \ta) \exp\{ - U(\mb{x})\}}{c_v(\mb{x}^{v \to \ta}, \mb{x}(v)) \exp\{-U(\mb{x}^{v \to \ta})\}} = \frac{ \exp\{-\Phi_v(\mb{x}^{v \to \ta})\} \exp \{ - U(\mb{x})\}}{\exp\{-\Phi_v(\mb{x})\} \exp \{ - U(\mb{x}^{v \to \ta})\}} = 1 . \qedhere \]
	\end{proof}
	
	Using these two properties of $\mathfrak{A}$, and the fact that $\zeta_t$ has full support, we see that
	\begin{align*}
		\frac{d}{dt} \big[ \zeta_t(U) - \shent(\zeta_t) \big]
			&= \sum_{\mb{x}, \mb{y}} \mathfrak{A}(\mb{x}, \mb{y}) \zeta_t \{\mb{y}\} \log \frac{\zeta_t\{\mb{x}\}}{P(\mb{x})} \\
			&= \sum_{\mb{x}, v, \ta} F_0 \left( \frac{P(\mb{x})}{P(\mb{x}^{v \to \ta})} \frac{\zeta_t \{\mb{x}^{v \to \ta}\}}{\zeta_t \{\mb{x}\}}\right) \zeta_t \{\mb{x}\} c_v(\mb{x}, \ta) .
	\end{align*}
	The desired formula follows from the fact that $\frac{P(\mb{x})}{P(\mb{x}^{v \to \ta})} = \frac{ \exp\{-\Phi_v(\mb{x})\}}{\exp\{-\Phi_v(\mb{x}^{v \to \ta})\}}$, also used above.
\end{proof}

We first use the previous result to show that free energy density is nonincreasing.

\begin{prop}
\label{prop:nonincreasing}
	Suppose $\mu \in \Prob(\A^\Gamma)$, and let $\Sigma = \big(\sigma_n \in \Hom(\Gamma, \Sym(V_n)) \big)_{n \in \NN}$ such that $\Delta^{\sigma_n} \to 0$.
	
	Then $\fed_\Sigma(\mu_0) \geq \fed_\Sigma(\mu_t)$ for all $t \geq 0$.
\end{prop}
\begin{proof}
	If $\fed_\Sigma(\mu_0) = +\infty$ then the result is trivial, so suppose $\fed_\Sigma(\mu_0) < +\infty$. This means that for any $\calO \ni \mu_0$ we have $\mb\Omega(\sigma_n, \calO) \ne \varnothing$ for all large enough $n$.
	
	Let $\calU_t$ be an arbitrary weak-open neighborhood of $\mu_t$. By Proposition \ref{prop:stability} there exists $\calU_0 \ni \mu_0$ such that, for all large enough $n$, we have $\zeta_t \in \mb\Omega(\sigma_n,\calU_t)$ whenever $\zeta_0 \in \mb\Omega(\sigma_n, \calU_0)$. 
	
	Suppose $n$ is large enough that $\mb\Omega(\sigma_n, \calU_0) \ne \varnothing$. Since $F_0 \leq 0$, the previous proposition implies that for any $\zeta_0 \in \Prob(\A^{V_n})$ and any $t > 0$
		\[ \frac{d}{dt} \big[ \zeta_t(U) - \shent(\zeta_t) \big] \leq 0 . \]
	Therefore for any $t \geq 0$
		\[ \zeta_0(U) - \shent(\zeta_0) \geq \zeta_t(U) - \shent(\zeta_t), \]
	and hence
		\[ \inf_{ \zeta \in \mb\Omega(\sigma_n, \calU_0)} \big[ \zeta(U) - \shent(\zeta) \big] \geq \inf_{ \zeta \in \mb\Omega(\sigma_n, \calU_t)} \big[ \zeta(U) - \shent(\zeta) \big] . \]
	Now by definition of $\fed_\Sigma$ we have
	\begin{align*}
		\fed_\Sigma(\mu_0)
			&= \sup_{ \calO \ni \mu_0} \limsup_{n \to \infty} \frac{1}{\abs{V_n}}  \inf_{ \zeta \in \mb\Omega(\sigma_n, \calO)} \big[ \zeta(U) - \shent(\zeta) \big] \\
			&\geq \limsup_{n \to \infty} \frac{1}{\abs{V_n}}  \inf_{ \zeta \in \mb\Omega(\sigma_n, \calU_0)} \big[ \zeta(U) - \shent(\zeta) \big] \\
			&\geq \limsup_{n \to \infty} \frac{1}{\abs{V_n}}  \inf_{ \zeta \in \mb\Omega(\sigma_n, \calU_t)} \big[ \zeta(U) - \shent(\zeta) \big] .
	\end{align*}
	Taking the supremum over $\calU_t$ gives the result.
\end{proof}

By a more careful analysis we can get the following, which may be interpreted as semicontinuity of the time derivative of $\fed_\Sigma(\mu_t)$ as a function of the measure:

\begin{prop}
\label{prop:strictdecrease}
	Suppose $\mu \in \Prob^\Gamma(\A^\Gamma)$ is not Gibbs, and let $\Sigma = \big( \sigma_n \in \Hom(\Gamma, \Sym(V_n)) \big)_{n \in \NN}$ be such that $\Delta^{\sigma_n} \to 0$. Then there exist $\delta,T>0$ and an open neighborhood $\calO \ni \mu$ such that if $\mu_0 \in \calO$ then $\fed_\Sigma(\mu_0) \geq \fed_\Sigma(\mu_t) + \delta t$ for all $t \in [0,T]$.
\end{prop}
Here we take the convention that $+\infty+\delta t = +\infty$.

\begin{proof}
	Since $\mu$ is not Gibbs, there exists $R$ such that either $\mu_R$ does not have full support or $\Delta_\ta^R(\mu_R) < 0$ for some $\ta \in \A$. We will come back to these two cases later, but for now let $R$ be fixed so that one of them occurs. We may assume $R \geq 1$.
	
	Let
		\[ s = \min \left\{ \frac{\exp\{-\Phi_e(\mb{x})\}}{\sum_{\tb \in \A} \exp\{-\Phi_e(\mb{x}^{e \to \tb})\}} \st \mb{x} \in \A^{\ball{e}{1}} \right\} > 0 , \]
	and call $v \in V$ \emph{good} if $\ball[\sigma]{v}{R} \cong \ball[\Gamma]{e}{R}$, and let $V'$ be the set of such $v$.
	Then
		\[ \frac{d}{dt} \big[ \zeta_t(U) - \shent(\zeta_t) \big]
			\leq s \sum_{\substack{\mb{x}, v, \ta \\ v \text{ good}}} F_0 \left( \frac{\exp\{ - \Phi_v(\mb{x})\}}{\exp\{ - \Phi_v(\mb{x}^{v \to \ta})\}} \frac{\zeta_t \{\mb{x}^{v \to \ta}\}}{\zeta_t \{\mb{x}\}}\right) \zeta_t \{\mb{x}\} . \]
	Let $\widetilde{P_\zeta^{\sigma, R}} \in \Prob(\A^{\ball[\Gamma]{e}{R}})$ be given by
		\[ \widetilde{P_\zeta^{\sigma, R}} \{\mb{y}\} = \frac{1}{\abs{V'}} \sum_{\substack{ \mb{x}, v \\ v \text{ good} \\ \mb{x} \restriction_{\ball{v}{R}} = \mb{y}}} \zeta \{\mb{x}\} \]
	Note that $\widetilde{P_\zeta^{\sigma, R}}$ is close to the $\ball[\Gamma]{e}{R}$-marginal of $P_\zeta^{\sigma}$ in total variation distance if most vertices are good.
	Then, applying Jensen's inequality,
	\begin{align*}
		\frac{d}{dt} \big[ \zeta_t(U) - \shent(\zeta_t) \big]
			&\leq s \sum_{\ta \in \A} \sum_{\mb{y} \in \A^{\ball[\Gamma]{e}{R}}} \sum_{\substack{\mb{x}, v \\ v \text{ good} \\ \mb{x}\restriction_{\ball{v}{R}} = \mb{y} }} F_0 \left( \frac{\exp\{ - \Phi_v(\mb{x})\}}{\exp\{ - \Phi_v(\mb{x}^{v \to \ta})\}} \frac{\zeta_t \{\mb{x}^{v \to \ta}\}}{\zeta_t \{\mb{x}\}}\right) \zeta_t \{\mb{x}\} \\
			&\leq s \abs{V'} \sum_{\ta \in \A} \sum_{\mb{y} \in \A^{\ball[\Gamma]{e}{R}}} \widetilde{P_{\zeta_t}^{\sigma,R}} \{\mb{y}\} \cdot F_0 \left( \frac{\exp\{ - \Phi_e(\mb{y})\}}{\exp\{ - \Phi_e(\mb{y}^{v \to \ta})\}} \frac{\widetilde{P_{\zeta_t}^{\sigma,R}}\{\mb{y}^{e \to \ta}\}}{\widetilde{P_{\zeta_t}^{\sigma,R}}\{\mb{y}\}} \right) .
	\end{align*}
	
	If $\mu_R$ does not have full support, there exist $\mb{y} \in \A^{\ball{e}{R}}$, $v \in \ball{e}{R}$ and $\ta \in \A$ such that $\mu_R\{\mb{y}\} \ne 0$ but $\mu_R\{\mb{y}^{v \to \ta}\} = 0$. Using translation-invariance of $\mu$, we may assume $v = e$. Then
		\[ \mu_{R}\{\mb{y}\} \cdot F_0 \left( \frac{\exp\{ - \Phi_e(\mb{y})\}}{\exp\{ - \Phi_e(\mb{y}^{e \to \ta})\}} \frac{\mu_R \{\mb{y}^{e \to \ta}\}}{\mu_R \{\mb{y}\}}\right) = - \mu_R \{\mb{y}\} < 0 . \]
	By continuity of $F_0$, there exists $\varepsilon>0$ such that 
		\[  \abs{a - \mu_{R}\{\mb{y}\}} < \varepsilon,\ 0 \leq b < \varepsilon  \ \Rightarrow \ a \cdot F_0 \left( \frac{\exp\{ - \Phi_e(\mb{y})\}}{\exp\{ - \Phi_e(\mb{y}^{e \to \ta})\}} \frac{b}{a}\right) < -\frac{\mu_R\{\mb{y}\}}{2} . \]
	In particular, if $\norm{\widetilde{P_{\zeta_t}^{\sigma,R}} - \mu_R}_{\TV} < \varepsilon$ then
		\[ \frac{d}{dt} \big[ \zeta_t(U) - \shent(\zeta_t) \big] < - s \abs{V'} \frac{\mu_R\{\mb{y}\}}{2} . \]
	Let $\calO_1 = \{ \nu \in \Prob(\A^\Gamma) \st \norm{\nu_R - \mu_R}_{\TV} < \varepsilon \}$. By continuity of $(\mu_0,t) \mapsto \mu_t$, we can pick $\calO, T$ such that if $\mu_0 \in \calO$ then $\mu_t \in \calO_1$ for all $t \in [0,T]$.
	
	Fix $\mu_0 \in \calO$. By Proposition \ref{prop:stability}, for any $\eta > 0$ there exists $\calU \ni \mu_0$ such that for all large enough $n$ if $\zeta_0 \in \mathbf{\Omega}(\sigma_n,\calU)$ then $\bar{d}(P_{\zeta_t}^{\sigma_n}, \mu_t) < \eta$ for all $t \in [0,T]$. If $\eta$ is small enough, this implies $\zeta_t \in \mathbf{\Omega}(\sigma_n,\calO_1)$ for all $t \in [0,T]$. Hence for all large enough $n$ if $\zeta_0 \in \mathbf{\Omega}(\sigma_n,\calU)$ then for any $t \in [0,T]$
		\[ \big[ \zeta_t(U) - \shent(\zeta_t) \big] - \big[ \zeta_0(U) - \shent(\zeta_0) \big] \leq - s \abs{V_n'} \frac{\mu_R\{\mb{y}\}}{2} t \]
	so
	\begin{align*}
		\big[ \zeta_0(U) - \shent(\zeta_0) \big]
			&\geq \big[ \zeta_t(U) - \shent(\zeta_t) \big] + s \abs{V_n'} \frac{\mu_R\{\mb{y}\}}{2} t \\
			&\geq \inf_{\zeta \in \mathbf{\Omega}(\sigma_n, \ball{\mu_t}{\eta})} \big[ \zeta(U) - \shent(\zeta) \big] + s \abs{V_n'} \frac{\mu_R\{\mb{y}\}}{2} t.
	\end{align*}
	Then, using that $\abs{V_n'}/\abs{V_n} \to 1$, and that the first limit in the definition of $\fed_\Sigma$ is a supremum
	\begin{align*}
		\fed_\Sigma(\mu_0)
			&\geq \limsup_{n \to \infty} \inf_{\zeta_0 \in \mathbf{\Omega}(\sigma_n, \calU)} \frac{1}{\abs{V_n}}\big[ \zeta_0(U) - \shent(\zeta_0)\big] \\
			& \geq \limsup_{n \to \infty} \inf_{\zeta \in \mathbf{\Omega}(\sigma_n, \ball{\mu_t}{\eta})} \frac{1}{\abs{V_n}}\big[ \zeta(U) - \shent(\zeta) \big] + s \frac{\mu_R\{\mb{y}\}}{2} t .
	\end{align*}
	Taking $\eta$ to zero gives the result in this case, with $\delta = s \frac{\mu_R\{\mb{y}\}}{2}$.
	
	If instead $\mu_R$ does have full support, then $\Delta_\ta^R(\mu_R) < 0$ for some $\ta$, so we can pick $\mb{y}$ such that 
		\[ \mu_{R}\{\mb{y}\} \cdot F_0 \left( \frac{\exp\{ - \Phi_e(\mb{y})\}}{\exp\{ - \Phi_e(\mb{y}^{e \to \ta})\}} \frac{\mu_R \{\mb{y}^{e \to \ta}\}}{\mu_R \{\mb{y}\}}\right) < 0  \]
	and proceed in the same way.
\end{proof}

\begin{proof}[Proof of Theorem]
	Suppose for the sake of contradiction that $\mu_t$ does not converge to the set of Gibbs measures; then we can pick some $\nu \in \Prob^\Gamma(\A^\Gamma)$ which is a limit point of $\{ \mu_t \st t \geq 0 \}$ but not a Gibbs measure.
		
	By Proposition~\ref{prop:strictdecrease}, we can pick $\delta, T>0$ and an open neighborhood $\calO \ni \nu$ such that for every $t$ with $\mu_t \in \calO$ we have $\fed_\Sigma(\mu_t) \geq \fed_\Sigma(\mu_{t+T}) + \delta T$.
	
	On the other hand, under the assumption that $\fed_\Sigma(\mu_s) < +\infty$, the set $\{ \fed_\Sigma(\mu_t) \st t \geq s \}$ is bounded: an upper bound is $\fed_\Sigma(\mu_s)$ by Proposition~\ref{prop:nonincreasing}, and a lower bound is $u^{\min} - \log \abs{\A}$. This is a contradiction, so the theorem follows.
\end{proof}


\section{Connection to property PA}
\label{sec:PAequiv}

We first prove the following:

\begin{prop}
	A group $\Gamma$ has property \PA\ if and only if for any $\mu \in \Prob^\Gamma(\A^\Gamma)$ there exists a sequence $\Sigma = \big( \sigma_n \in \Hom(\Gamma, \Sym(V_n)) \big)_{n \in \NN}$ and a sequence $(\mb{x}_n \in \A^{V_n})_{n \in \NN}$ with $P_{\mb{x}_n}^{\sigma_n} \xrightarrow{wk} \mu$ and $\Delta^{\sigma_n} \to 0$.
\end{prop}

Some ideas for this proof were shared with me by Lewis Bowen.

\begin{proof}
	The `if' direction is clear, since each $P_{\mb{x}_n}^{\sigma_n}$ is periodic.
	
	For the other direction, suppose $\Gamma$ has property \PA\ and let $\mu \in \Prob^\Gamma(\A^\Gamma)$. By definition of property \PA, we can pick a sequence of periodic measures $(\mu_n)_{n \in \NN}$ converging to $\mu$.
	
	Fix $n \in \NN$. The support of $\mu_n$ consists of finitely many orbits under $\Gamma$; let $\{\mb{y}_1, \ldots, \mb{y}_k\} \subset \A^\Gamma$ be a set which contains exactly one element of each orbit, and denote the (finite) orbits by $\Gamma \mb{y}_i$. Then we can write
		\[ \mu_n = \sum_{i=1}^k a_i \Unif( \Gamma \mb{y}_i) . \]
	
	Pick natural numbers $m_1, \ldots, m_k$, and let $V_n$ be the disjoint union of $m_i$ copies of $ \Gamma \mb{y}_i$ for each $i$. Let $\sigma_n \in \Hom(G, \Sym(V_n))$ act separately on each copy of each orbit. Let $\mb{x}_n \in \A^{V_n}$ be given by
		\[ \mb{x}_n(v) = v(e). \]
	Then
		\[ P_{\mb{x}_n}^{\sigma_n} = \sum_{i=1}^k \frac{m_i}{\sum_{j=1}^k m_j} \Unif(\Gamma \mb{y}_i) , \]
	so if $m_1, \ldots, m_k$ are chosen appropriately then we can ensure $P_{\mb{x}_n}^{\sigma_n} \to \mu$.
	
	Now we need to show that we can ensure $\Delta^{\sigma_n} \to 0$. Let $\nu \in \Prob(\{0,1\}^\Gamma)$ be the product measure with uniform base. Then then above argument implies the existence of sequences $\big( \sigma_n \in \Hom(\Gamma, \Sym(V_n)) \big)_{n \in \NN}$ and $\mb{z}_n \in (\A \times \{0,1\})^{V_n}$ with $P_{\mb{z}_n}^{\sigma_n} \to \mu \times \nu$. If we write $\mb{z}_n = (\mb{x}_n, \mb{y}_n)$ with $\mb{x}_n \in \A^{V_n}$ and $\mb{y}_n \in \{0,1\}^{V_n}$, then $P_{\mb{x}_n}^{\sigma_n} \to \mu$ and $P_{\mb{y}_n}^{\sigma_n} \to \nu$. We will show that the latter fact implies $\Delta^{\sigma_n} \to 0$.	
	Suppose $v \in V_n$, $\gamma \in \Gamma$ are such that $\sigma_n^\gamma v = v$. Then for any $\beta \in \Gamma$,
		\[ \big( \Pi_v^{\sigma_n} \mb{y}_n \big)(\beta \gamma) = \mb{y}_n \big( \sigma_n^{\beta \gamma} v \big) = \mb{y}_n \big( \sigma_n^{\beta} v \big) = \big( \Pi_v^{\sigma_n} \mb{y}_n \big)(\beta) . \]
	In particular, for any finite set $D \subset \Gamma$ we have
		\[ P_{\mb{y}_n}^{\sigma_n} \{ \mb{w} \in \{0,1\}^\Gamma \st \mb{w}(\beta \gamma) = \mb{w}(\beta) \ \forall \beta \in D \}
			\geq \frac{1}{\abs{V_n}} \abs{\{v \in V_n \st \sigma_n^\gamma v = v\}} . \]
	But by assumption, as $n \to \infty$ the left-hand side converges to
	\begin{align*}
		\nu \{ \mb{w} \in \{0,1\}^\Gamma \st \mb{w}(\beta \gamma) = \mb{w}(\beta) \ \forall \beta \in D \}
		 	&\leq \nu \{ \mb{w} \st \mb{w}(\beta \gamma) = \mb{w}(\beta) \ \forall \beta \in D \text{ s.t. } \beta\gamma \not\in D \} \\
			&= 2^{-\abs{D \gamma \setminus D}} ,
	\end{align*}
	and hence
		\[ \limsup_{n \to \infty} \frac{1}{\abs{V_n}} \abs{\{v \in V_n \st \sigma_n^\gamma v = v\}} \leq 2^{-\abs{D \gamma \setminus D}} . \]
	As long as $\gamma \neq e$, the set $D$ can be chosen to make $\abs{D \gamma \setminus D}$ arbitrarily large, so that
		\[ \lim_{n \to \infty} \frac{1}{\abs{V_n}} \abs{\{v \in V_n \st \sigma_n^\gamma v = v\}} = 0 . \]
		
	For any $R \in \NN$, it can be checked that if $\sigma_n^\gamma v \neq v$ for all $\gamma \neq e$ such that $\abs{\gamma} \leq 2R +1$ then the map
	\begin{align*}
		\ball[\Gamma]{e}{R} &\to \ball[\sigma]{v}{R} \\
		\gamma &\mapsto \sigma^\gamma v
	\end{align*}
	is an isomorphism of the (labeled, directed) induced subgraphs. Therefore
		\[ \delta_R^{\sigma_n} \leq \sum_{\gamma \in \ball[\Gamma]{e}{2R+1} \setminus \{e\}} \frac{1}{\abs{V_n}} \abs{\{v \in V_n \st \sigma_n^\gamma v = v\}} \to 0 , \]
	which, since $R$ is arbitrary, implies $\Delta^{\sigma_n} \to 0$.
\end{proof}

\subsection{Proof of Proposition \ref{prop:PAequiv}}
Note $\fed_\Sigma(\mu) = +\infty$ if and only if there exists an open neighborhood $\calO \ni \mu$ such that $\bOmega(\sigma_n, \calO)$ is empty for infinitely many $n$. Therefore if $\fed_\Sigma(\mu) < +\infty$ there exists a sequence $\zeta_n \in \Prob(\A^{V_n})$ with $P_{\zeta_n}^{\sigma_n} \to \mu$. Since each $P_{\zeta_n}^{\sigma_n}$ is periodic, this shows that if for every $\mu$ there exists $\Sigma$ with $\fed_\Sigma(\mu) < +\infty$ then $\Gamma$ has property \PA.

Conversely, if $\Gamma$ has property \PA\ and $\mu \in \Prob^\Gamma(\A^\Gamma)$ is given, then by the above proposition we can pick $\Sigma$ and $(\mb{x}_n)_{n \in \NN}$ with $\Delta^{\sigma_n} \to 0$ and $P_{\mb{x}_n}^{\sigma_n} \to \mu$. But then for any open $\calO \ni \mu$ we have $\delta_{\mb{x}_n} \in \bOmega(\sigma_n, \calO)$ for all large enough $n$, so $\fed_\Sigma(\mu) < +\infty$.

\section{Proofs of statements involving infinitary dynamics}
\label{sec:infdynamics}
We first give some additional setup regarding the Glauber dynamics on $\A^\Gamma$. First, recall that on an infinite graph we must first define the Markov generator on a `core' of `smooth' functions. Let $C(\A^\Gamma)$ denote the space of continuous real-valued functions on $\A^\Gamma$, with the supremum norm $\norm{\cdot}_\infty$.
The smooth functions are defined as follows: given $f \in C(\A^\Gamma)$ and $v \in \Gamma$, let
	\[ \Delta_f(v) = \sup \{ \abs{f(\eta_1) - f(\eta_2)} \st \eta_1(u) = \eta_2(u) \ \forall u \ne v \} , \]
	\[ \nnorm{f} = \sum_{v \in \Gamma} \Delta_f(v) , \]
and
	\[ D(\A^\Gamma) = \{ f \in C(\A^\Gamma) \st \nnorm{f} < \infty \} . \]
Every function which depends on only finitely many coordinates is in $D(\A^\Gamma)$, so $D(\A^\Gamma)$ is dense in $C(\A^\Gamma)$. For every $f \in D(\A^\Gamma)$ the series defining $\Omega f$ converges absolutely, and $\Omega f \in C(\A^\Gamma)$.

Note that the condition $\nnorm{f} < \infty$ does not imply that $f$ is continuous; in fact for every tail-measurable $f$ we have $\nnorm{f} = 0$.

Continuing to follow mostly the notation from Liggett's book, let
	\[ c_u(v) = \sup \big\{ \norm{c_u(\eta_1, \cdot) - c_u(\eta_2, \cdot)}_{TV} \st \eta_1(\gamma) = \eta_2(\gamma) \ \forall \gamma \ne v \big\} . \]
Then
	\[ \Theta \beta (u) = \sum_{v \in \Gamma} \beta(v) c_u (v) \]
defines a bounded linear operator on $\ell^1(\Gamma)$.

The closure $\bar\Omega^\Gamma$ is a Markov generator, so its domain is a dense subset of the continuous functions $C(\A^\Gamma)$ and the range of $I-\lambda \bar\Omega^\Gamma$ is all of $C(\A^\Gamma)$ for all $\lambda \geq 0$. We also have $\norm{f} \leq \norm{(I-\lambda \bar\Omega^\Gamma)f}$ for all $\lambda \geq 0$ \cite[comment after Definition 2.1]{liggett2005}. In particular $I-\lambda \bar\Omega^\Gamma$ is injective. An important consequence is that we have a contraction $(I - \lambda \bar\Omega^\Gamma)^{-1} \colon C(\A^\Gamma) \to C(\A^\Gamma)$.

\subsection{Approximate equivariance (Proof of Theorem \ref{thm:Sequivariance})}

For $\beta \in \RR^\Gamma$, let
	\[ \norm{\beta} = \sup_{\gamma \in \Gamma} \abs{\beta(\gamma)} (3r)^{\abs{\gamma}} . \]

\begin{lemma}
\label{lem:nnormlipbound}
	For any continuous $g \colon \A^\Gamma \to \RR$,
		\[ \frac{1}{3} \nnorm{g} \leq \norm*{\Delta_g} = \abs{g}_{\Lip} . \]
\end{lemma}
In particular, every Lipschitz function is in $D(\A^\Gamma)$.

\begin{proof}
	For the inequality:
		\[ \nnorm{g} = \sum_{\gamma} \Delta_g(\gamma) \leq \left(\sup_{\gamma} \Delta_g(\gamma) (3r)^{\abs{\gamma}} \right) \sum_{\gamma} (3r)^{-\abs{\gamma}} \leq 3 \sup_{\gamma} \Delta_g(\gamma) (3r)^{\abs{\gamma}} .\]
	Now similarly, for any $\mathbf{x},\mathbf{y} \in \A^\Gamma$
	\begin{align*}
		\abs{g(\mathbf{x}) - g(\mathbf{y})} &\leq \sum_{\gamma \st \mathbf{x}(\gamma) \ne \mathbf{y}(\gamma)} \Delta_g(\gamma) \tag{using continuity} \\
		&\leq \sup_{\gamma} \Delta_g(\gamma) (3r)^{\abs{\gamma}} \sum_{\gamma \st \mathbf{x}(\gamma) \ne \mathbf{y}(\gamma)} (3r)^{-\abs{\gamma}} \\
		&= \sup_{\gamma} \Delta_g(\gamma) (3r)^{\abs{\gamma}} \cdot d(\mathbf{x},\mathbf{y}),
	\end{align*}
	so
		\[ \abs{g}_{\Lip} \leq \sup_{\gamma} \Delta_g(\gamma) (3r)^{\abs{\gamma}} = \norm{\Delta_g} . \]
	The converse inequality follows from the fact that for any $\gamma \in \Gamma$
		\[ \Delta_g(\gamma) \leq (3r)^{-\abs{\gamma}} \abs{g}_{\Lip} . \qedhere \]
\end{proof}

\begin{lemma}
	With respect to the norm $\norm{\cdot}$ on $\RR^\Gamma$, $\Theta$ has operator norm at most
		\[ M \coloneqq \sup_{\gamma} \sum_{h \in \Gamma} c_h(\gamma) (3r)^{d(h, \gamma)} < \infty. \]
\end{lemma}
\begin{proof}
	For any $\gamma \in \Gamma$,
	\begin{align*}
		[\Theta \beta](\gamma) \cdot (3r)^{\abs{\gamma}}
			&\leq \sum_{h \in \Gamma} \abs{\beta(h)} c_h(\gamma) (3r)^{\abs{\gamma}} \\
			&\leq \sum_{h \in \Gamma} \abs{\beta(h)} c_h(\gamma) (3r)^{\abs{h} + d(h, \gamma)} \\
			&\leq \norm{\beta} \sum_{h \in \Gamma} c_h(\gamma) (3r)^{d(h, \gamma)} ,
	\end{align*}
	so, taking the supremum over $\gamma$, we see that $\norm{\Theta \beta} \leq \norm{\beta} M$ and hence the operator norm is bounded by $M$.

	Finiteness of $M$ follows from the fact that always $c_h(\gamma) \leq 2$, and $c_h(\gamma) = 0$ if $h,\gamma$ are not adjacent. So for any $\gamma$
		\[ \sum_{h \in \Gamma} c_h(\gamma) (3r)^{d(h,\gamma)} \leq 2 \cdot 2r \cdot (3r)^1 = 12 r^2 . \qedhere \]
\end{proof}

We can now give a proof of Lemma \ref{lem:misc}:

\begin{proof}[Proof of Lemma \ref{lem:misc}]
By \cite[Theorem 3.9(c)]{liggett2005},
	\[ \Delta_{S^\Gamma(t) f} \leq \exp(t \Theta) \Delta_{f} \quad \forall f \in D(\A^\Gamma) . \]
Taking $\norm{\cdot}$ norms of both sides gives, by Lemma~\ref{lem:nnormlipbound},
	\[ \abs{S^\Gamma(t) f}_{\Lip} \leq \exp(Mt) \abs{f}_{\Lip} . \]
The result follows from this and the definition of $\bar{d}$.
\end{proof}

\begin{prop}
\label{prop:Lipschitzbound}
For all small enough $\lambda>0$, for all $g \in D(\A^\Gamma)$ we have
	\[ \abs{(I - \lambda\bar\Omega)^{-k}g}_{\Lip} \leq [1 - \lambda M]^{-k} \abs{g}_{\Lip} .\]
\end{prop}
\begin{proof}
Recall from \cite[proof of Theorem 3.9]{liggett2005} that for all small enough $\lambda > 0$
	\[ \Delta_{(I - \lambda \bar\Omega)^{-k}g} \leq [(1+\lambda\varepsilon)I - \lambda \Theta]^{-k} \Delta_g \]
for any $g \in D(\A^\Gamma)$. If we apply the $\norm{\cdot}$ norm to both sides we get, by Lemma \ref{lem:nnormlipbound},
	\[ \abs{(I - \lambda\bar\Omega)^{-k}g}_{\Lip} \leq [1 - \lambda (M-\varepsilon)]^{-k} \abs{g}_{\Lip} .\]
The stated bound follows after dropping $\varepsilon$, which is positive.
\end{proof}

Define
	\[ \nnorm{f}_R = \sum_{\abs{\gamma} \geq R} \Delta_f (\gamma) . \]
If $f \in D(\A^\Gamma)$ then $\lim_{R \to \infty} \nnorm{f}_R = 0$; if $f$ is Lipschitz then for any $R \geq 0$
	\[ \nnorm{f}_R \leq 3 \abs{f}_{\Lip} (2/3)^{R} . \]
The following result establishes an approximate equivariance of $P_{\mb{x}}^\sigma$ with the generator:
\begin{prop}
\label{prop:Qequivariance}
	Let $V$ be a finite set and $\sigma \in \Hom(\Gamma, \Sym(V))$. For any $f \in D(\A^\Gamma)$, $R \in \NN$, and ${\mb{x}} \in \A^V$, 
		\[ \abs*{\Omega^\sigma P_{\mb{x}}^\sigma f - P_{\mb{x}}^\sigma \Omega^\Gamma f} \leq 3 \nnorm{f}_R + 2 \delta_R^\sigma \nnorm{f}. \]
\end{prop}

\begin{proof}
	From the definitions of $\Omega^\sigma$ and $P_{\mb{x}}^\sigma$,
	\begin{align*}
		\Omega^\sigma P_{\mb{x}}^\sigma f
			&= \sum_{v \in V} \sum_{\ta \in \A} c_v(\mathbf{x}, \ta) \big[ P_{\mathbf{x}^{v \to \ta}}^\sigma f - P_{\mb{x}}^\sigma f \big] \\
			&= \frac{1}{\abs{V}} \sum_{v,w \in V} \sum_{\ta \in \A} c_v(\mathbf{x}, \ta) \big[ f(\Pi_w^\sigma \mb{x}^{v \to \ta}) - f(\Pi_w^\sigma \mathbf{x}) \big] .
	\end{align*}
	We can compare this to if the sum is restricted to pairs $v,w$ which are nearby in the graph $\sigma$:
	\begin{align*}
		&\hspace{-2cm} \abs*{\Omega^\sigma P_{\mb{x}}^\sigma f - \frac{1}{\abs{V}} \sum_{w \in V} \sum_{v \in \ball[\sigma]{w}{R-1}} \sum_{\ta \in \A} c_v(\mb{x}, \ta) \big[ f(\Pi_w^\sigma \mathbf{x}^{v \to \ta}) - f(\Pi_w^\sigma \mathbf{x}) \big] } \\
			&\leq \frac{1}{\abs{V}} \sum_{w \in V} \sum_{v \not\in \ball[\sigma]{w}{R-1}} \sum_{\ta \in \A} c_v(\mathbf{x}, \ta) \abs{ f(\Pi_w^\sigma \mathbf{x}^{v \to \ta}) - f(\Pi_w^\sigma \mathbf{x}) }
	\end{align*}
	Now since the labelings $\mb{x}^{v \to \ta}$ and $\mb{x}$ differ only at $v$, their lifted labelings $\Pi_w^\sigma \mb{x}^{v \to \ta}$ and $\Pi_w^\sigma \mb{x}$ differ only at preimages of $v$ under the map $\gamma \mapsto \sigma^\gamma w$. Let $\Pi_w^\sigma \{v\} \subset \Gamma$ denote the set of these preimages. Then the above is bounded by
		\[ \frac{1}{\abs{V}} \sum_{w \in V} \sum_{v \not\in \ball[\sigma]{w}{R-1}} \sum_{\ta \in \A} c_v(\mathbf{x}, \ta) \sum_{\gamma \in \Pi_w^\sigma \{v\}} \Delta_f(\gamma) = \frac{1}{\abs{V}} \sum_{w \in V} \sum_{v \not\in \ball[\sigma]{w}{R-1}} \sum_{\gamma \in \Pi_w^\sigma \{v\}} \Delta_f(\gamma) . \]
	Since for each $w$ the sets in the collection $\{ \Pi_w^\sigma \{v\} \st v \not\in \ball[\sigma]{w}{R-1}\}$ are disjoint and contained in the complement of $\ball[\Gamma]{e}{R-1}$, we can bound this by
			\[ \frac{1}{\abs{V}} \sum_{w \in V} \sum_{\gamma \not\in \ball[\Gamma]{e}{R-1}} \Delta_f(\gamma) = \nnorm{f}_R. \]
	
	Now suppose $w \in V$ is such that $\ball[\sigma]{w}{R} \cong \ball[\Gamma]{e}{R}$: then for each $v \in \ball[\sigma]{w}{R-1}$ the intersection $\ball[\Gamma]{e}{R-1} \cap \Pi_w^\sigma \{v\}$ consists of a single point, which we call $\gamma^v$. We then have $c_v(\mb{x}, \ta) = c_{\gamma^v}(\Pi_w^\sigma \mb{x}, \ta)$. From this we can get 
	\begin{align*}
		&\hspace{-2cm} \Bigg\lvert \sum_{v \in \ball[\sigma]{w}{R-1}} \sum_{\ta \in \A} c_v(\mathbf{x}, \ta) \big[ f(\Pi_w^\sigma \mb{x}^{v \to \ta}) - f(\Pi_w^\sigma \mb{x}) \big] \\[-1em]
			&\qquad - \sum_{\gamma \in \ball[\Gamma]{e}{R-1}} \sum_{\ta \in \A} c_{\gamma}(\Pi_w^\sigma \mb{x}, \ta) \big[ f\big( (\Pi_w^\sigma \mb{x})^{\gamma \to \ta}\big) - f\big(\Pi_w^\sigma \mb{x}\big) \big] \Bigg\rvert \tag{$\ast$} \\[0.5em]
		&= \abs*{\sum_{v \in \ball[\sigma]{w}{R-1}} \sum_{\ta \in \A} c_v(\mb{x}, \ta)[f(\Pi_w^\sigma \mb{x}^{v \to \ta}) - f\big( (\Pi_w^\sigma \mb{x})^{\gamma^v \to \ta}\big) ] } \\[0.5em]
		&\leq \sum_{v \in \ball[\sigma]{w}{R-1}} \sum_{\ta \in \A} c_v(\mb{x}, \ta) \abs*{ f(\Pi_w^\sigma \mb{x}^{v \to \ta}) - f\big( (\Pi_w^\sigma x)^{\gamma^v \to \ta}\big) }
	\intertext{Now our construction also guarantees that the labelings $\Pi_w^\sigma \mb{x}^{v \to \ta}$ and $(\Pi_w^\sigma \mb{x})^{\gamma^v \to \ta}$ differ only at sites in $\Pi_w^\sigma \{v\}$ other than $\gamma^v$, all of which lie outside $\ball[\Gamma]{e}{R-1}$. Therefore we can continue}
		(\ast) &\leq \sum_{v \in \ball[\sigma]{w}{R-1}} \sum_{\ta \in \A} c_v(\mb{x}, \ta) \sum_{\substack{\gamma \in \Pi_w^\sigma \{v\} \\ \abs{\gamma} \geq R}} \Delta_f(\gamma) \\
		&\leq \sum_{\abs{\gamma} \geq R} \Delta_f(\gamma) \\
			&= \nnorm{f}_R .
	\end{align*}
	In the second-to-last line, we have again used that $\Pi_w^\sigma\{v_1\}$ and $\Pi_w^\sigma \{v_2\}$ are disjoint if $v_1 \neq v_2$.
	
	For other `bad' $w$ where $\ball[\sigma]{w}{R} \not\cong \ball[\Gamma]{e}{R}$, approximating the sum over $v$ by the sum over $\gamma$ in this way may be inaccurate, but the fraction of $w\in V$ which are bad is only $\delta_R^\sigma$. For these $w$ we note that the magnitudes of both sums in ($\ast$) can be bounded by $\sum_{\gamma \in \Gamma} \Delta_f(\gamma) = \nnorm{f}$.
	
	So far we have shown that
		\[ \abs*{ \Omega^\sigma P_{\mb{x}}^\sigma f - \frac{1}{\abs{V}} \sum_{w \in V} \sum_{\abs{\gamma} < R} \sum_{\ta \in \A} c_{\gamma}(\Pi_w^\sigma \mb{x}, \ta) \big[ f\big( (\Pi_w^\sigma \mb{x})^{\gamma \to \ta}\big) - f\big(\Pi_w^\sigma \mb{x}\big) \big]} \leq 2 \nnorm{f}_R + 2 \delta_R^\sigma \nnorm{f} . \]
	To finish, we compare the second term on the left-hand side to $P_{\mb{x}}^\sigma \Omega^\Gamma f$ using an approach similar to above:
	\begin{align*}
		&\hspace{-2cm}\Bigg\lvert \frac{1}{\abs{V}} \sum_{w \in V} \sum_{\abs{\gamma} < R} \sum_{\ta \in \A} c_{\gamma}(\Pi_k^\sigma \mb{x}, \ta) \big[ f\big( (\Pi_k^\sigma \mb{x})^{\gamma \to \ta}\big) - f\big(\Pi_k^\sigma \mb{x}\big) \big] \\[-1em]
		&\qquad - \frac{1}{\abs{V}} \sum_{w \in V} \sum_{\gamma \in \Gamma} \sum_{\ta \in \A} c_{\gamma}(\Pi_w^\sigma \mb{x}, \ta) \big[ f\big( (\Pi_w^\sigma \mb{x})^{\gamma \to \ta}\big) - f\big(\Pi_w^\sigma \mb{x}\big) \big] \Bigg\rvert \\[0.5em]
		&\leq \frac{1}{\abs{V}} \sum_{w \in V} \sum_{\abs{\gamma} \geq R} \sum_{\ta \in \A} c_{\gamma}(\Pi_w^\sigma \mb{x}, \ta) \abs*{ f\big( (\Pi_w^\sigma \mb{x})^{\gamma \to \ta}\big) - f\big(\Pi_w^\sigma \mb{x}\big)} \\
		&\leq \frac{1}{\abs{V}} \sum_{w \in V} \sum_{\abs{\gamma} \geq R} \Delta_f(\gamma) \\
		&= \nnorm{f}_R. \qedhere
	\end{align*}
\end{proof}

\begin{lemma}
\label{lem:Sequivariance}
	For all small enough $\lambda>0$, for any $m \in \NN$ and $g \in \Lip(\A^\Gamma)$ we have
		\[ \norm{ (I - \lambda \Omega^\sigma)^{-m} P_{\mb{x}}^\sigma g - P_{\mb{x}}^\sigma (I - \lambda \bar\Omega^\Gamma)^{-m} g}_{\ell^\infty(\A^V)} \leq \lambda \Delta^\sigma \abs{g}_{\Lip} \sum_{k=1}^m (1-\lambda M)^{-k} . \]
\end{lemma}
\begin{proof}
	We use induction on $m$, starting with the base case $m=1$. Throughout, we assume $\lambda$ is small enough for Proposition \ref{prop:Lipschitzbound} to apply.

	Given $g \in \Lip(\A^\Gamma)$, let $f = (I-\lambda \bar\Omega^\Gamma)^{-1} g$.
	Then for any $R \in \NN$
	\begin{align*}
		\norm{P_{\mb{x}}^\sigma g - (I - \lambda \Omega^\sigma)[P_{\mb{x}}^\sigma f]}_{\ell^\infty(\A^V)}
			&= \norm{P_{\mb{x}}^\sigma [f - \lambda \bar\Omega^\Gamma f] - (I - \lambda \Omega^\sigma)[P_{\mb{x}}^\sigma f]}_{\ell^\infty(\A^V)} \\
			&= \lambda \norm{\Omega^\sigma P_{\mb{x}}^\sigma f - P_{\mb{x}}^\sigma \bar\Omega^\Gamma f}_{\ell^\infty(\A^V)} \\
			&\leq \lambda (3 \nnorm{f}_R + 2 \delta_R^\sigma \nnorm{f}) \tag{Prop. \ref{prop:Qequivariance}}\\
			&\leq \lambda (9 \cdot (2/3)^{R} + 6\delta_R^\sigma) \abs{f}_{\Lip}.
	\intertext{Taking the infimum over $R$ gives}
		\norm{P_{\mb{x}}^\sigma g - (I - \lambda \Omega^\sigma)[P_{\mb{x}}^\sigma f]}_{\ell^\infty(\A^V)}
			&\leq \lambda \Delta^\sigma \abs{f}_{\Lip} \\
			&\leq \lambda \Delta^\sigma (1-\lambda M)^{-1} \abs{g}_{\Lip} . \tag{Prop. \ref{prop:Lipschitzbound}}
	\end{align*}
	Since $(I - \lambda \Omega^\sigma)^{-1}$ is a contraction on $\ell^\infty(\A^V)$,
	\begin{align*}
		\norm{ (I - \lambda \Omega^\sigma)^{-1} P_{\mb{x}}^\sigma g - P_{\mb{x}}^\sigma (I - \lambda \bar\Omega^\Gamma)^{-1} g}_{\ell^\infty}
		&= \norm{(I - \lambda \Omega^\sigma)^{-1} \left[P_{\mb{x}}^\sigma g - (I - \lambda \Omega^\sigma)[P_{\mb{x}}^\sigma f] \right]}_{\ell^\infty} \\
		&\leq \lambda \Delta^\sigma (1-\lambda M)^{-1} \abs{g}_{\Lip} .
	\end{align*}
	This proves the base case.
	
	Now assuming the $m$ case and the base case, we prove the $m+1$ case:
	\begin{align*}
		&\hspace{-1cm}\norm{ (I - \lambda \Omega^\sigma)^{-(m+1)} P_{\mb{x}}^\sigma g - P_{\mb{x}}^\sigma (I - \lambda \bar\Omega^\Gamma)^{-(m+1)} g}_{\ell^\infty} \\
			&= \big\| (I - \lambda \Omega^\sigma)^{-1}\left[(I - \lambda \Omega^\sigma)^{-m} P_{\mb{x}}^\sigma g - P_{\mb{x}}^\sigma (I - \lambda \bar\Omega^\Gamma)^{-m} g\right] \\
			&\quad + \left[(I - \lambda \Omega^\sigma)^{-1} P_{\mb{x}}^\sigma (I - \lambda \bar\Omega^\Gamma)^{-m}g - P_{\mb{x}}^\sigma (I - \lambda \bar\Omega^\Gamma)^{-1}(I - \lambda \bar\Omega^\Gamma)^{-m} g\right] \big\|_{\ell^\infty} \\
		&\leq \norm{(I - \lambda \Omega^\sigma)^{-m} P_{\mb{x}}^\sigma g - P_{\mb{x}}^\sigma (I - \lambda \bar\Omega^\Gamma)^{-m} g}_{\infty} \tag{contraction}\\
		&\quad + \lambda \Delta^\sigma (1-\lambda M)^{-1} \abs{(I - \lambda \bar\Omega^\Gamma)^{-m}g}_{\Lip} \tag{base case}\\
		&\leq \lambda \Delta^\sigma \abs{g}_{\Lip} \sum_{k=1}^{m+1} (1-\lambda M)^{-k} . \tag{inductive hyp., Prop. \ref{prop:Lipschitzbound}}
	\end{align*}
	This completes the induction.
\end{proof}

\subsubsection*{Proof of Theorem \ref{thm:Sequivariance}}
	Given $g \in \Lip_1(\A^\Gamma)$, for all large enough $m$ we can apply the previous lemma with $\lambda = t/m$, which gives
		\[ \norm{ (I - \tfrac{t}{m} \Omega^\sigma)^{-m} P_{\mb{x}}^\sigma g - P_{\mb{x}}^\sigma (I - \tfrac{t}{m} \bar\Omega^\Gamma)^{-m} g}_{\ell^\infty(\A^V)} \leq \tfrac{t}{m} \Delta^\sigma \abs{g}_{\Lip} \sum_{k=1}^m (1-\tfrac{t}{m} M)^{-k} . \]
	Let $m \to \infty$. The left-hand side converges (by Hille-Yosida; \cite[Theorem 2.9(b)]{liggett2005}) to $\norm{S^\sigma(t) P_{\mb{x}}^\sigma g - P_{\mb{x}}^\sigma S^\Gamma(t) g}_\infty$ while the lim sup of the right-hand side is bounded by $\Delta^\sigma t e^{Mt} \abs{g}_{\Lip}$, since
		\[ \limsup_{m \to \infty} \frac{1}{m} \sum_{k=1}^{m} (1- \tfrac{t}{m} M)^{-k} \leq \limsup_{m \to \infty} \frac{1}{m} \sum_{k=1}^{m} (1- \tfrac{t}{k} M)^{-k} = e^{Mt}. \]
	Since $g \in \Lip(\A^\Gamma)$ was arbitrary, the inequality of transportation distance follows.

\bibliographystyle{alpha}
\bibliography{Metastability}

\end{document}